\documentclass[11pt,reqno]{amsart}   	
\usepackage{geometry}                		
\usepackage[utf8]{inputenc}
\usepackage{graphicx}
\usepackage[english]{babel}
\usepackage{amsmath}
\usepackage{amsfonts}
\usepackage{amsthm}
\usepackage{color}
\usepackage{float}
\usepackage{caption}
\usepackage{subcaption}
\usepackage[noend,algonl]{algorithm2e}
\SetKwProg{Fn}{Function}{}{end}

\newtheorem{proposition}{Proposition}
\newtheorem{theorem}{Theorem}
\newtheorem*{corollary}{Corollary}
\newtheorem*{remark}{Remark}

\newcommand{\K}{\mathcal{K}}
\graphicspath{{./}}

\title{Nonlinearity helps the convergence of the Inverse Born Series}

\author{Nicholas DeFilippis}
\address{Courant Institute of Mathematical Sciences}
\email{nad9961@nyu.edu}

\author{Shari Moskow}
\address{Department of Mathematics, Drexel University, Philadelphia, PA, USA }
\email{moskow@math.drexel.edu}

\author{John C. Schotland }
\address{Department of Mathematics and Department of Physics, Yale University, New Haven, CT, USA}
\email{john.schotland@yale.edu}

\date{}

\begin{document}

\maketitle
\begin{abstract} 
In previous work of the authors, we investigated the Born and inverse Born series for a scalar wave equation with linear and nonlinear terms, the nonlinearity being cubic of Kerr type~\cite{DeMoSc}. We reported conditions which guarantee convergence of the inverse Born series, enabling recovery of the coefficients of the linear and nonlinear terms. In this work, we show that if the coefficient of the linear term is known, an arbitrarily strong Kerr nonlinearity can be reconstructed, for sufficiently small data. Additionally, we show that similar convergence results hold for general polynomial nonlinearities. Our results are illustrated with numerical examples. 
\end{abstract}

\section{Introduction}
There has been considerable recent interest in inverse problems
for nonlinear partial differential equations (PDEs)~\cite{assylbekov_1,assylbekov_2,carstea,imanuvilov,isakov_1,isakov_2,isakov_3,lassas,isakov_4,KrUh,kang,kurylev,lassas}.  Applications of such problems arise in a variety of contexts, including optical imaging and seismology. Similar to the case of linear PDEs, the goal is to recover an unknown spatially varying coefficient from boundary measurements. The above referenced works have demonstrated that nonlinearity is of great utility in proving uniqueness of the solution to the inverse problem for a large class of nonlinear PDEs. We note that reconstruction methods have also been developed~\cite{barreto,carstea,griesmaier,kaltenbacher,lassas_2023,DeMoSc}. 
In this paper, we show that nonlinearity is also helpful for reconstruction, in the sense that in certain cases, an arbitrarily strong nonlinearity can be recovered for sufficiently small scattering data.

In previous work \cite{DeMoSc}, we considered the inverse problem of recovering the coefficients of a nonlinear elliptic PDE arising in the study of the Kerr effect. The Kerr effect is a nonlinear optical process that leads to focusing or defocusing of light~\cite{boyd}. In ~\cite{DeMoSc} the unknowns to be reconstructed are the coefficients of both a linear term and a cubic term in the PDE. To this end, we constructed the Born series and found a recursive formula for the forward operators arising in the series.  We also obtained bounds on the forward operators and gave conditions which guarantee convergence of the inverse born series (IBS). The IBS was then used to reconstruct both coefficients from boundary measurements. Although the IBS has been extensively applied to inverse problems for linear PDEs~\cite{review}, Ref.~\cite{DeMoSc} was the first report of its use for a nonlinear PDE. 

{\color{black} In this paper, we consider a variant of the above the inverse problem in which the coefficient of the linear term is known a priori. 
Surprisingly, we find that not reconstructing the linear term leads to several advantages. First, it is possible to find explicit bounds on the forward operators in the Born series. In contrast, in our previous work \cite{DeMoSc}, the bounds on the forward operators were not explicit. Second, we show that when reconstructing the coefficient of the cubic term, the IBS converges if the boundary data is sufficiently small. This finding is strikingly different than the case of the linear inverse problem, where the IBS series fails to converge for sufficiently large data. Finally, our results extend to the case of general polynomial nonlinearities. These include second and third harmonic generation, which affords a much greater range of physical applications~\cite{boyd}. We note that the linear response of a scattering medium can, in principle, be acquired by means of hole burning experiments, in which the nonlinear response is suppressed~\cite{boyd}.}

The paper is organized as follows. In section 2, we restate the forward problem and the fixed point convergence result for small data with a known linear term. We then describe the forward Born series in section 3, where we also find explicit expressions for the bounds on the forward operators. In section 4 we state the convergence results for the IBS, where we show that small data leads to an arbitrarily large radius of convergence for the IBS. The case of more general polynomial nonlinearities is treated in section 5. Section 6 contains numerical reconstructions for a two-dimensional medium. Our conclusions are presented in section 7. 

\section{Forward problem}

We consider a bounded domain $\Omega$ in $\mathbb{R}^d$ with a smooth boundary, for $d\ge 2$. The scalar field $u$, which for the Kerr effect, obeys the nonlinear PDE
\begin{align}
\label{baseequation}
\Delta u + k^2 u+ k^2\beta(x) |u|^2  u &= 0 \quad \text{ in } \quad \Omega \ , \\
\frac{\partial u}{\partial \nu } &= g \quad  \text{ on } \quad \partial\Omega \ ,
\end{align} 
where the wavenumber $k$ is real and $\nu$ is the unit outward normal to $\partial\Omega$. The coefficient $\beta$ is the nonlinear susceptibility~\cite{boyd}, which we assume is real valued and $g$ is a boundary source. It follows that $u$ is real valued, so that $|u|^2u= u^3$. More generally, $u$ is complex valued, in which case the results hold with small modifications.

\begin{remark}
Here, for simplicity, we have assumed that the coefficient of the linear term is constant. If the coefficient of the linear term is not constant, our results carry over by modifying the associated Green's function as explained below.
\end{remark}

To proceed, we require the solution $u_0$ to the  linear problem
\begin{align}
\label{backgroundequation}
    \Delta u_0 + k^2u_0 &= 0 \quad \text{ in } \quad \Omega \ , \\
    \frac{\partial u_0}{\partial \nu } &= g   \quad  \text{ on } \quad \partial\Omega 
    \end{align}
which we assume throughout this paper is well posed, that is, that $-k^2$ is not a Neumann eigenvalue of the Laplacian on $\Omega$. Following standard procedures \cite{CoKr}, we find that the field $u$ obeys the integral equation
\begin{equation}\label{integralequation}
    u(x) = u_0(x) - k^2\int_\Omega G(x, y)  \beta(y)u^3(y)dy \ .
\end{equation} 
where the Green's function $G$ obeys
\begin{align}
	\Delta_x G(x, y) + k^2 G(x, y)  &= \delta(x-y)  \quad \text{ in } \quad \Omega \ , \\
	 \frac{\partial G }{\partial \nu_y} &= 0   \quad  \text{ on } \quad \partial\Omega \ .
\end{align}
 We define the nonlinear operator $T: C(\overline{\Omega})\rightarrow C(\overline{\Omega})$ by
\begin{equation}\label{Tdef}
    T(u) = u_0 - k^2\int_\Omega G(x, y) \beta(y) u^3(y) dy.
\end{equation}
Note that if $u\in C(\overline{\Omega})$ is a fixed point of $T$, then $u$ satisfies equation (\ref{integralequation}). 
The following result provides conditions for existence of a unique solution to (\ref{integralequation}) within a ball in $C(\overline{\Omega})$. 
\begin{proposition}
\label{bounds}  
Let $T: C(\overline{\Omega})\rightarrow C(\overline{\Omega})$ be defined by (\ref{Tdef}) and define $\mu$ by 
\begin{equation}\label{mudef1} 
 \mu = k^2\sup_{x\in \Omega} \int_\Omega | G(x,y) | dy. 
 \end{equation}
If $$\Vert\beta\Vert_\infty < \frac{4}{27\mu \Vert u_0\Vert_{C(\overline{\Omega})}^2},$$ then $T$ has a unique fixed point on the ball  of radius ${\| u_0\|_{{\color{black} C(\overline{\Omega})}}/{2}}$ about $u_0$ in $C(\overline{\Omega})$, {\color{black} and fixed point iteration starting with $u_0$ converges in $C(\Omega)$ to the unique fixed point $u$}.
\end{proposition}
The proof is given in the Appendix of \cite{DeMoSc}. We note that this shows that given any bounded $\beta$, the fixed point iteration will converge for small enough $u_0$; that is, for small enough data $g$. 
 Hence the same is true for the forward Born series \cite{DeMoSc}. However, the fixed point analysis does not provide bounds on the forward operators, and therefore does not provide information about the convergence of the inverse Born series. 
\section{Born series}
The forward problem is to compute the field $u$ on $\partial\Omega$ given a prescribed source $g$ on $\partial\Omega$. For the inverse problem, we will consider a set of sources $g$ where each is associated with a boundary point $x\in\partial\Omega$, and in this manner view $u$ and $u_0$ as functions in $C(\Omega \times \partial \Omega)$. 
A series representing the solution of the forward problem is derived by iteration of the integral equation (\ref{integralequation}), beginning with the background field $u_0$. By doing this, one can show that we obtain
\begin{equation} \label{born_series}
    \phi = K_1 (\beta) + K_2 (\beta, \beta) + K_3( \beta , \beta,  \beta )+\cdots \ ,
\end{equation} 
where $\phi=u-u_0$ is the data on the boundary. In \cite{DeMoSc}, we found that the forward operator $$K_n: [L^\infty(\Omega)]^{n}\rightarrow C(\partial\Omega\times\partial\Omega)$$ is a $n$-linear operator (multilinear of order $n$) and given by the recursive formula 
\begin{eqnarray} K_0 &=& u_0 ,\nonumber \\   K_1 &=& B u_0\otimes u_0\otimes u_0 \nonumber, \\ 
K_{n+1} &=&   B\sum_{ \substack{ (i_1, i_2, i_3) \\ i_1+i_2+i_3 =n \\ 0\leq i_1,i_2,i_3 \leq n}} K_{i_1}\otimes K_{i_2}\otimes K_{i_3}. \label{Kformula} \end{eqnarray} where
the tensor operator $B$ takes a multilinear operator of order $l$ to one of order $l+1$
$$ BT_l(\beta_1, \ldots , \beta_l, \beta_{l+1} )= b( T_l (\beta_1, \ldots , \beta_l ) , \beta_{l+1} ), $$
and the operator 
$b: C(\overline{\Omega}) \times [ L^\infty(\Omega)]\rightarrow C(\overline{\Omega})$ is given by 
\begin{equation} \label{beq} b(v, \beta  ) = k^2 \int_{\Omega} G(x, y) \beta (y) v(y) dy . \end{equation} 
In the above definition we also used the tensor product of multilinear operators.
Given $T_j$ and $T_l$, which are multilinear operators of order $j$ and $l$, respectively, the tensor product $T_l\otimes T_j$ is defined by
$$ T_l\otimes T_j (\beta_1, \ldots , \beta_l, \beta_{l+1},\dots, \beta_{l+j} )=  T_l (\beta_1, \ldots , \beta_l ) T_j (\beta_{l+1},\dots, \beta_{l+j} ),$$
so that $T_l\otimes  T_j$ is a multilinear operator of order $l+j$.  See \cite{DeMoSc} for a proof that fixed point iterations generate the series (\ref{born_series}) with operators given by  (\ref{Kformula}). 
We will refer to this series as the (forward) Born series.
We note that Proposition~\ref{bounds} guarantees convergence of the forward Born series.

In order to analyze the convergence of the inverse Born series, bounds on the norms of the forward operators $K_i$ are required.  For any multilinear operator $K$ of order $n$ on $[L^\infty(\Omega)]^{n}$,  if we define 
  $$ | K |_\infty = \sup_{\substack{ \beta_1,\ldots,\beta_n \neq 0 }}  { \| K(\beta_1,\dots\beta_n) \|_{{\color{black} C(\partial \Omega \times \partial \Omega)}} \over{ \| \beta_1\|_\infty \cdots\| \beta_n\|_\infty }} ,  $$
  then we have the following boundedness result. 
  \begin{proposition} \label{forwardopbounds} The forward operator $K_n$, given by (\ref{Kformula}) is a bounded multilinear operator from $[L^\infty(\Omega)]^{n}$ to $C(\partial{\Omega}\times\partial\Omega),$  and \begin{equation}\label{Knbound} | K_n |_\infty  \leq \nu( K \mu)^n ,\end{equation}
 where \begin{equation}\label{mudef}  \mu = k^2\sup_{x\in \Omega} \int_\Omega | G(x,y) | dy , \end{equation} 
  $$\nu={3\over{2}} \| u_0\|_{C(\overline{\Omega}\times\partial\Omega)},$$ and $$K= {27\over{4}}  \| u_0\|_{C(\overline{\Omega}\times\partial\Omega)}^2.$$ 
\end{proposition}
\begin{proof}
From Lemma 2 of \cite{DeMoSc}, the forward operators $K_n$, as defined by (\ref{Kformula}),  are bounded multilinear operators from $[L^\infty(\Omega)]^{n}$ %to $C(\overline{\Omega}\times\partial\Omega)$ and $C(\partial\Omega\times\partial\Omega)$, and in both cases 
 to $C(\partial\Omega\times\partial\Omega)$ and satisfy 
 \begin{equation}\label{Knbound} | K_n |_\infty  \leq \nu_n \mu^{n} ,  \end{equation}
 where \begin{equation}\label{mudef}  \mu = k^2\sup_{x\in \Omega} \int_\Omega | G(x,y) | dy, \end{equation}
 $$ \nu_0 = \| u_0 \|_{C(\overline{\Omega}\times\partial\Omega )}, $$
 and for all $n\geq 0$,
 \begin{equation} \label{nudef} \nu_{n+1} = \sum_{ \substack{  (i_1,i_2,i_3) \\ i_1+i_2+i_3 =n \\ 0\leq i_1,i_2,i_3 \leq n}} \nu_{i_1}\nu_{i_2}\nu_{i_3} . \end{equation}
 We therefore need to show that  the sequence $\{ \nu_n\} $ defined by (\ref{nudef}), for any $n\geq 0$, satisfies $$ {\nu_n} \leq \nu K^n .$$
where $K = {27\nu_0^2\over{4}}$ and $\nu = {3\over{2}}\nu_0$.
We  proceed as in \cite{DeMoSc} and consider the generating function 
 $$ P(x) = \sum_{n=0}^\infty \nu_n x^n .$$ 
From \cite{DeMoSc} we know that this power series has a positive radius of convergence; here we repeat the argument while finding the radius explicitly.   Computing the cube of $P$,
\begin{eqnarray} (P(x))^3 &=& \sum_{ \substack{  i_1,i_2,i_3  } } x^{i_1} x^{i_2} x^{i_3}  \nu_{i_1}\nu_{i_2}\nu_{i_3} \nonumber \\
&=& \sum_{n=0}^\infty f_n x^n \nonumber , \end{eqnarray}
where $$ f_n = \sum_{ \substack{ (i_1,i_2,i_3) \\  i_1+i_2+i_3 =n \\ 0\leq i_1,i_2,i_3 \leq n}} \nu_{i_1}\nu_{i_2}\nu_{i_3} ,$$
and multiply  (\ref{nudef}) by $x^n$ and sum to obtain
 $$ \sum_{n=0}^\infty \nu_{n+1} x^n  =   \sum_{n=0}^\infty f_n x^n,$$
which yields
  \begin{equation}\label{polynomial}  x(P(x))^3  - P(x) + \nu_0 =0. \end{equation}
We differentiate with respect to $x$ to
  \begin{equation} \label{ode} P^\prime (x) = - {(P(x))^3 \over{3x(P(x))^2 -1 }} \end{equation}
  with $P(0) = \nu_0$. 
Just as was argued in \cite{DeMoSc}, this equation must have an analytic solution on an open interval around $x=0$, and this solution must be the series $P(x)$. 

Now, since $P(x)>0$ for $x>0$, (\ref{ode}) implies that $P$ is increasing for $x>0$, so long as $3x(P(x))^2< 1$.  Algebraic manipulation of (\ref{polynomial}) gives  $$  3xP^2 -1 = -{3\over{P}}\nu_0 + 2, $$
so as long as  $3xP^2 -1 <0$, $P< {3\over{2}} \nu_0 $, and the series converges.  We can see that this is true for any $0\leq x < {4/{(27\nu_0^2)}}$, since in the above equation $P< {3\over{2}}\nu_0$ when $3x({3\over{2}}\nu_0)^2 < 1$. We have therefore shown that for any $0\leq x < {4/{(27\nu_0^2)}}$, the terms of the series must tend to zero as $n\rightarrow\infty$, and in particular must be bounded by some $\nu$. Since the entire series sum is always bounded by ${3\over{2}} \nu_0$ and the terms are all positive, we may take $\nu = {3\over{2}} \nu_0$. 
So, for all $n$ we have that  
$$ \nu_n \leq  {3\over{2}} \nu_0 \left({1\over{x}}\right)^n, $$ and this holds for any  $0< x < {4/{(27\nu_0^2)}}$. Hence we must have $$ \nu_n \leq  {3\over{2}} \nu_0 \left({27\nu_0^2\over{4}}\right)^n. $$  \end{proof}
We note that by majorizing the series by a geometric series, these bounds give another proof of convergence of the forward series with the same requirements as Proposition \ref{bounds}.
 \begin{corollary} The Born series
 $$ u = u_0 + \sum_{n=1}^\infty K_n (\beta,\ldots,\beta) \ , $$
  where $K_n$ are given by (\ref{Kformula}), converges in  $C(\overline{\Omega})$ for 
 $$ \| \beta\|_\infty <  {1\over{K\mu}} $$
 where $K$ and $\mu$ are given as in Proposition \ref{forwardopbounds}.
\end{corollary}

\section{Inverse Born Series}
The inverse problem is to reconstruct the coefficient $\beta$ from measurements of the scattering data {\color{black} $\phi= u -u_0 $ on $\partial \Omega$}, and we propose to do this by computing the inverse Born series (IBS) \cite{review}, which is defined as 
\begin{equation}\label{inversedefinition}
    \tilde{\beta} = \mathcal{K}_1 \phi + \mathcal{K}_2 (\phi)  + \mathcal{K}_3( \phi)  + \cdots \ ,
\end{equation}
where the data $\phi  \in C(\partial\Omega\times\partial\Omega).$ The IBS was analyzed in \cite{moskow_1,toSc}. The inverse operators $\mathcal{K}_m$  are given by
\begin{align}
\label{inv_operators}
\mathcal{K}_1 (\phi) &= K_1^{+} (\phi),\\
\mathcal{K}_2(\phi) &=-\mathcal{K}_1\left(K_2 (\mathcal{K}_1(\phi),\mathcal{K}_1(\phi))\right),\\
\mathcal{K}_m(\phi) &= -\sum_{n=2}^{m}\sum_{i_1+\cdots+i_n = m}  \K_1{K}_n \left( \mathcal{K}_{i_1}(\phi), \dots, \mathcal{K}_{i_n}(\phi) \right) ,
\label{inv_operators_again}
\end{align}
where $K_1^+$ is some regularized pseudoinverse of $K_1$. 

The bounds on the forward operators in Proposition \ref{forwardopbounds} allow us to apply Theorem 2.2  and Theorem 2.4 of \cite{toSc}.  We note that the constants $\nu$ and $\mu$ in \cite{toSc} correspond to $\nu K \mu$ and $K\mu$ here in Proposition \ref{forwardopbounds}. We denote by $ \| \mathcal{K}_1 \|$ the operator norm of $\mathcal{K}_1$ as a map from $C(\partial\Omega\times\partial\Omega)$ to $L^\infty(\Omega)$, and use $\| u_0\|$ to refer to the $C(\overline{\Omega}\times\partial\Omega )$ norm. Theorems 2.2 and 2.4 of \cite{toSc} yield the following results. 
\begin{theorem}[Convergence of the inverse Born series] \label{thm:conv_inv}
If   $\Vert \mathcal{K}_1 \phi\Vert_\infty < r $, where  the radius of convergence $r$ is given by
$$
r=\frac{2}{27\mu\| u_0\|^2 } \left[\sqrt{16 C^2+1}-4 C \right],
$$
with $C = \max\{2,{81\over{8}}\mu\|\mathcal{K}_1\| \| u_0\|^3 \}$ and $\mu$ given by (\ref{mudef1}),  then the inverse Born series (\ref{inversedefinition}) converges. 
\end{theorem}
\begin{theorem}[Approximation error]
\label{thm:error}
Suppose that the hypotheses of Theorem~\ref{thm:conv_inv} hold and that the Born and inverse Born series converge. Let $\tilde\beta$ denote the sum of the inverse Born series.  Setting $\mathcal{M} = \max\left\{\|\beta \|_\infty,\|\tilde\beta\|_\infty\right\},$ if we further assume that
\begin{eqnarray}
\label{M_bound}
\mathcal{M}  < \frac{4}{27\mu \| u_0\|^2 }\left(1-\sqrt{\frac{{81\over{8}}\mu\|\mathcal{K}_1\| \| u_0\|^3 }{1+{81\over{8}}\mu\|\mathcal{K}_1\| \| u_0\|^3}}\right),
\end{eqnarray}
%then the approximation error can be estimated as follows:
%\begin{multline}
%\label{eq:combined_error}
%\nonumber
%\left\| \beta - \sum_{m=1}^N \mathcal{K}_m(\phi) \right\|_\infty  \le M\left( \frac{\|{\beta_1}\|_\infty}{r}\right)^{N+1} \frac{1}{1- \frac{\|{\beta_1}\|_\infty}{r}} \\ 
%+ \left( 1 - \frac{{81\over{8}}\mu\|\mathcal{K}_1\| \| u_0\|^3}{(1-{27\over{4}}\mu\| u_0\|^2\mathcal{M})^2}+{81\over{8}}\mu\|\mathcal{K}_1\| \| u_0\|^3\right)^{-1}\left\| (I-\mathcal{K}_1K_1)\beta\right\|_\infty,
%\end{multline} 
then the error of the series sum can be estimated 
\begin{equation}
\label{eq:combined_error}
\nonumber
\left\| \beta - \tilde\beta \right\|_\infty  \le \left( 1 - \frac{{81\over{8}}\mu\|\mathcal{K}_1\| \| u_0\|^3}{(1-{27\over{4}}\mu\| u_0\|^2\mathcal{M})^2}+{81\over{8}}\mu\|\mathcal{K}_1\| \| u_0\|^3\right)^{-1}\left\| (I-\mathcal{K}_1K_1)\beta\right\|_\infty.
\end{equation} 
%
%where 
%\begin{equation*}
%\label{eq:M_eq}
% M =\frac{{27\over{2}}\mu\|u_0\|^2}{\sqrt{16 C^2+1}}.
% % \frac{1}{\mu}\left( 1- \sqrt{\frac{\nu\|\mathcal{K}_1\|}{1+\nu \|\mathcal{K}_1\|}}\right)
%\end{equation*}
\end{theorem}
{\color{black} Note that if this were a well posed problem, and $\mathcal{K}_1$ were a true inverse of $K_1$, Theorem \ref{thm:error}  says that the inverse series would converge to the true $\beta$ under these hypotheses.  Due to the need for regularization, the right hand side in the conclusion of Theorem \ref{thm:error} is nonzero in general.  If one scales $u_0$ (or equivalently the boundary data) by some constant $\gamma$, $K_1$ will exactly scale by $\gamma^3$. So, we can choose its pseudoinverse $\mathcal{K}_1$  to scale by ${1/{\gamma^3}}$. Hence the quantity $ {81\over{8}}\mu\|\mathcal{K}_1\| \| u_0\|^3 $  will remain fixed, and Theorem \ref{thm:conv_inv} implies that the radius $r$ will grow arbitrarily large as $\gamma\rightarrow 0$. Furthermore, in this case Theorem \ref{thm:error}  says  that the error in the series sum is bounded by a constant times $\left\| (I-\mathcal{K}_1K_1)\beta\right\|$ for $\| u_0\|$ small enough, which is the error introduced with the (necessary) regularization. The error in the tail of the series can be bounded geometrically, see \cite{DeMoSc} for details.}

\section{General polynomial nonlinearities}
We now consider the case of general polynomial nonlinearities without a linear term.  We consider the PDE 
\begin{align}
\label{generalpdeequation}
\Delta u + k^2 u+ k^2 \sum_{l=2}^L \beta^{(l)}(x) u^l  &= 0 \quad \text{ in } \quad \Omega \ , \\
\frac{\partial u}{\partial \nu } &= g \quad  \text{ on } \quad \partial\Omega \ ,
\end{align} 
where the unknown coefficients to be reconstructed are 
$ \vec{\beta} = ( \beta^{(2)},\ldots,\beta^{(L)})$.  We similarly obtain the forward operators 
\begin{eqnarray} K_0 &=& u_0 ,\nonumber \\   K_1 &=& \sum_{l=2}^L B^{(l)} u_0\otimes \ldots \otimes u_0 \nonumber, \\ 
K_{n+1} &=&  \sum_{l=2}^L  B^{(l)} \sum_{ \substack{ (i_1, \ldots, i_l) \\ i_1+\ldots +i_l =n \\ 0\leq i_1,\ldots,i_l \leq n}} K_{i_1}\otimes \ldots \otimes K_{i_l}. \label{generalKformula} \end{eqnarray} where
all of the tensor operators of order $p$ now input a list of $p$ vectors; where $B^{(l)}$ now extracts the entry corresponding to the  $l$ power, 
$$ B^{(l)}T(\vec{\beta}_1, \ldots , \vec{\beta}_q, \vec{ \beta}_{q+1} )= b^{(l)}( T(\vec{\beta}_1, \ldots , \vec{\beta}_q) ,  \vec{\beta}_{q+1} ) \ $$
where 
$b^{(l)}: C(\overline{\Omega}) \times [ L^\infty(\Omega)]\rightarrow C(\overline{\Omega})$ is given by 
\begin{equation} \label{bleq} b^{(l)}(v, \vec{\beta}  ) = k^2 \int_{\Omega} G(x, y) \beta^{(l)}(y) v(y) dy . \end{equation} 
One bounds $K_n$ in a similar manner to obtain (\ref{Knbound}) where now 
 \begin{equation} \label{nudefgeneral} \nu_{n+1} =  \sum_{l=2}^L  \sum_{ i_1+\ldots +i_l =n } \nu_{i_1}\ldots\nu_{i_l} . \end{equation}
again with  $$ \nu_0 = \| u_0 \|_{C(\overline{\Omega}\times\partial\Omega )}.$$ The generating function for this sequence 
$$P(x) =\sum_{i=1}^\infty \nu_i x^i $$
satisfies  \begin{equation}\label{genpoly} xQ(P(x)) -P(x) +\nu_0 =0 \end{equation}
where
$$ Q(x) = \sum_{l=2}^L x^l .$$ Differentiating this expression, we find that $P$ is analytic in a neighborhood of zero and is increasing for $x>0$, until $x Q^\prime(P(x)) =1$. Using that $Q^\prime(P) \leq {L\over{P}} Q(P) $, both are increasing, and that  $$ x {LQ(P)\over{P}} -1 = L-1 -{L\nu_0\over{P}}$$ from the polynomial (\ref{genpoly}), we deduce that $P$ is analytic while $P < {L\nu_0 \over{L-1}}$.  Hence $P$ is analytic when $$ x< { \nu_0 \over{ (L-1) Q({L\nu_0\over{L-1}})}}.$$ Proposition \ref{forwardopbounds} therefore holds in the general case with $$ \nu = {L\nu_0\over{L-1}} $$ and $$ K =(L-1) {Q({L\nu_0\over{L-1}}) \over{\nu_0}},$$ 
with again $$\nu_0 = \| u_0\|.$$ {\color{black} That is,
the forward operator $K_n$, given by (\ref{generalKformula}), is a bounded multilinear operator from $[L^\infty(\Omega)]^{(L-1)n}$ to $C(\partial{\Omega}\times\partial\Omega)$  and \begin{equation}\label{generalKnbound} | K_n |_\infty  \leq \nu( K \mu)^n ,\end{equation}
 where  \begin{equation}\label{mudef}  \mu = k^2\sup_{x\in \Omega} \int_\Omega | G(x,y) | dy , \end{equation} 
  $$\nu={L\over{L-1}} \| u_0\|_{C(\overline{\Omega}\times\partial\Omega)},$$ and $$K= (L-1){Q({L\over{L-1}} \| u_0\|) \over{  \| u_0\|}}.$$
 The forward series will converge if $K\mu <1 $,  and we clearly have $K\leq C \| u_0 \|^s $ with $s \geq 1$ for $\| u_0\|$ small, since 
we assumed that the polynomial $Q$ has degree greater than or equal to $2$.  In a similar manner, the inverse series will have radius 
$$r=\frac{1}{2K\mu} \left[\sqrt{16 C^2+1}-4 C \right],$$
where 
$C = \max\{2,\|\mathcal{K}_1\|\nu K\mu\}$.  Note that if we scale the data, the situation is similar to the cubic case, since $K\nu$ scales as $Q(\| u_0\|)$, while $\mathcal{K}_1$ scales as $ 1/ Q(\| u_0\|)$, so that $r\rightarrow\infty$ as $\| u_0\| \rightarrow 0$, due to the presence of $K$ in the denominator in the expression for $r$. }

\begin{remark} If the nonlinearity is polynomial in both $u$ and its complex conjugate $\overline{u}$, this analysis carries over, with the forward operators generalized to have conjugates appropriately placed in (\ref{generalKformula}). If there is at most one term per degree, the constants $\nu$ and $K$ will remain the same as in the real case presented here. If there is more than one term for some degree, the constants will need to be modified slightly; however, they will scale similarly with $\nu_0$. 
\end{remark}

\section{Numerical Reconstructions}
In this section, we present a few numerical simulations to demonstrate convergence of the IBS for high contrast. We note that the restriction to the real case and to two dimensions is for  simplicity and is not fundamental.  To generate synthetic data, we solve the nonlinear PDE
\begin{align}
\label{pde}
\Delta u + k^2 u + k^2\beta(x) u^3 &= 0 \quad \text{ in } \quad \Omega \ , \\
\frac{\partial u}{\partial \nu } &= g \quad  \text{ on } \quad \partial\Omega \ ,
\end{align}
and the background PDE
\begin{align}
\label{pde}
\Delta u_0 + k^2 u_0 &= 0 \quad \text{ in } \quad \Omega \ , \\
\frac{\partial u}{\partial \nu } &= g \quad  \text{ on } \quad \partial\Omega \ ,
\end{align}
by using a Galerkin finite element method as implemented in the FEniCS library in Python. The domain $\Omega$ is the unit disk, and we obtain the finite element mesh automatically in FEniCS. The boundary source $g$ is taken to be $g(x)=g_0\delta(x-y)$, where $y\in\partial\Omega$ and $g_0$ is the strength of the source. The delta function is approximated by a Gaussian for numerical computations, and we will force small $\| u_0\|$ by decreasing $g_0$. 
The forward operators $K_n$ are constructed according to the formulas (\ref{Kformula}), and the operator $B$, defined by the corresponding integral operator $b$ given in (\ref{beq}), is evaluated by solving the a background PDE source problem.  We use a different mesh to compute the forward operators from the one used to generate the boundary data. Note that these background problems are linear, and only the right-hand side of the PDE changes for each evaluation of  $b$.  The inverse Born series is implemented according to \eqref{inv_operators}--\eqref{inv_operators_again}. The solution to the linearized inverse problem is given in terms of the operator $\mathcal K_1$, which is constructed from a regularized pseudoinverse of the forward operator $K_1$. {\color{black} In our calculations we used the built in numpy \texttt{pinv} function,  which uses SVD and cuts the singular values below the ratio \texttt{rcond}, which we found we needed to choose between $\texttt{rcond}=10^{-6}$ and  $\texttt{rcond}=10^{-4}$.} In all of the following figures, we employ $16$ sources and $32$ detectors, and two frequency values $k=1,2$, each for $8$ of the sources. Only one value of $g_0$ is used per experiment in order to emphasize the effects of scaling. 

In Figure \ref{fig:highthreegaussians} we show an example of the reconstructions of three Gaussians of very high contrast, {\color{black} in this case over 20:1. } The sources were implemented with small $g_0=0.01$, and the series converged rapidly, with the first term already close to the projection $\mathcal{K}_1 K_1 \beta$.  One would expect that $\mathcal{K}_1 K_1\beta$ is the best one could hope for given the regularization. The cross section reveals the rapid convergence. 

In our next experiment, we take $\beta$ to be a disk of high contrast (5:1) with a jump against the background. We see the reconstructions in Figure \ref{fig:disk}, where we take the source scaling to be only moderately small, with $g_0=0.1$.  Here we see the higher order terms in the series improving the reconstruction. The shape of the disk is recovered even better than  $\mathcal{K}_1 K_1 \beta$. 

For the third and final experiment, we present a Gaussian and the disk side by side as seen in Figure \ref{fig:diskandgaussian}, with a moderate scaling $g_0=0.1$.  Again the higher order terms in the series improve the reconstruction, even differentiating the two inhomogeneities better than $\mathcal{K}_1 K_1 \beta$.

\begin{figure}[t]
    \centering
    \begin{subfigure}[]{1\textwidth}
        \centering
        \includegraphics[height=2.5in]{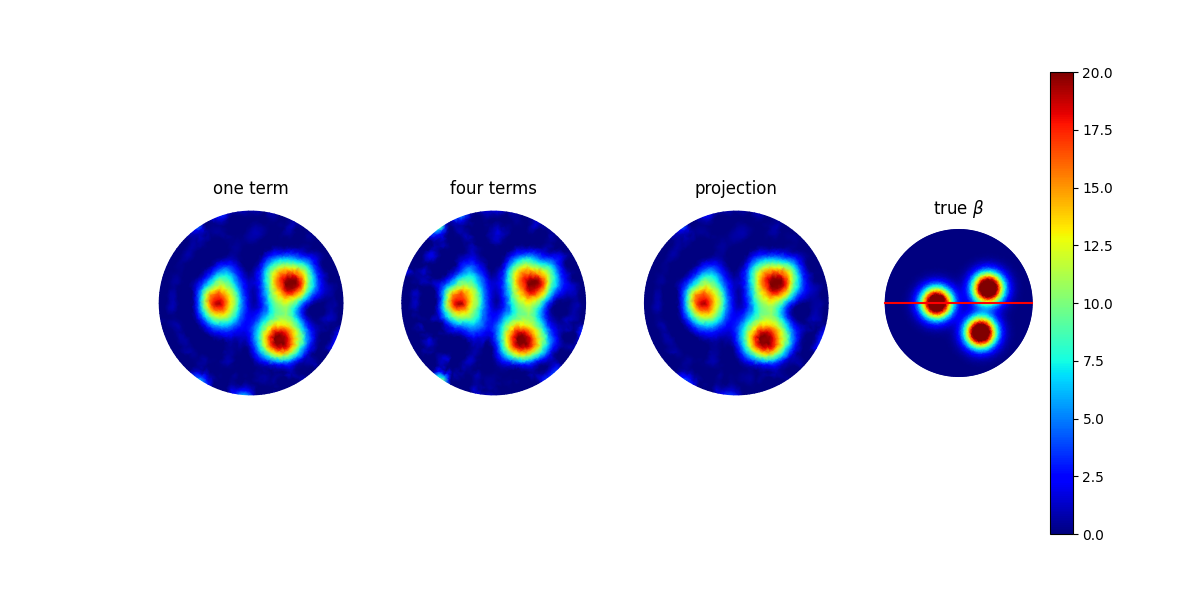}
        \caption{Reconstruction of  high contrast $\beta $.}
        \label{fig:sub1}
    \end{subfigure}
    \begin{subfigure}[]{1\textwidth}
        \centering
        \includegraphics[height=2.5in]{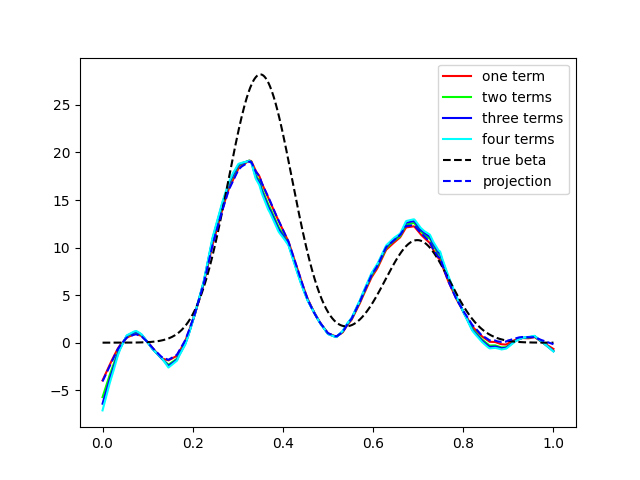}
        \caption{Cross section of reconstruction of high contrast $\beta$.}
        \label{fig:sub2}
    \end{subfigure}
    \caption{Reconstruction of a high contrast $\beta$. Sources were scaled down with $g_0=0.01$ to ensure convergence of the IBS. The projection of the true $\beta$ onto the regularization space, $\mathcal{K}_1 K_1 \beta$, indicates an expected best case scenario.}
    \label{fig:highthreegaussians}
\end{figure}

\begin{figure}[t]
    \centering
    \begin{subfigure}[]{1\textwidth}
        \centering
        \includegraphics[height=2.5in]{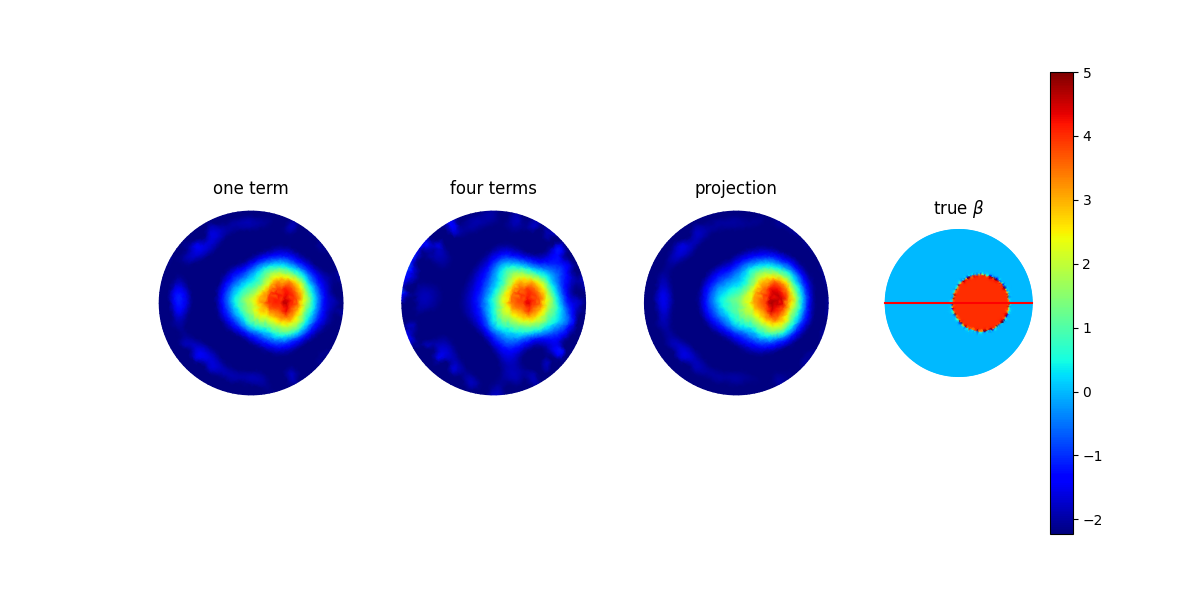}
        \caption{Reconstruction of disk shaped inclusion $\beta $.}
        \label{fig:sub1}
    \end{subfigure}
    \begin{subfigure}[]{1\textwidth}
        \centering
        \includegraphics[height=2.5in]{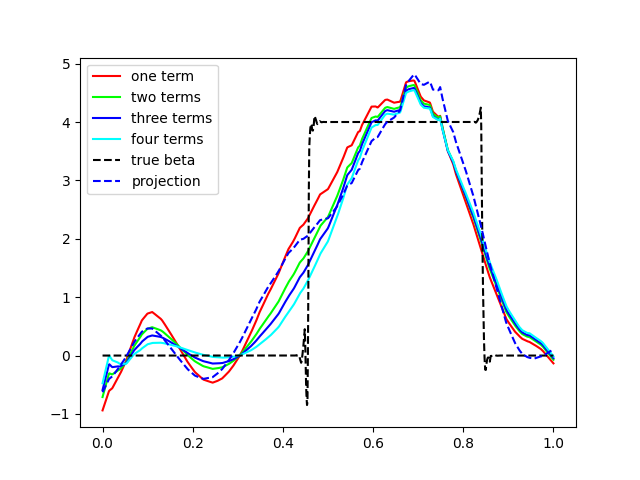}
        \caption{Cross section of reconstruction of disk inclusion $\beta$.}
        \label{fig:sub2}
    \end{subfigure}
    \caption{Reconstruction of $\beta$ with a discontinuous disk shaped inclusion. Sources were scaled moderately with $g_0=.1$. The series captures the shape better than the projection $\mathcal{K}_1 K_1 \beta$.}
    \label{fig:disk} 
    \end{figure}

\begin{figure}[t]
    \centering
    \begin{subfigure}[]{1\textwidth}
        \centering
        \includegraphics[height=2.5in]{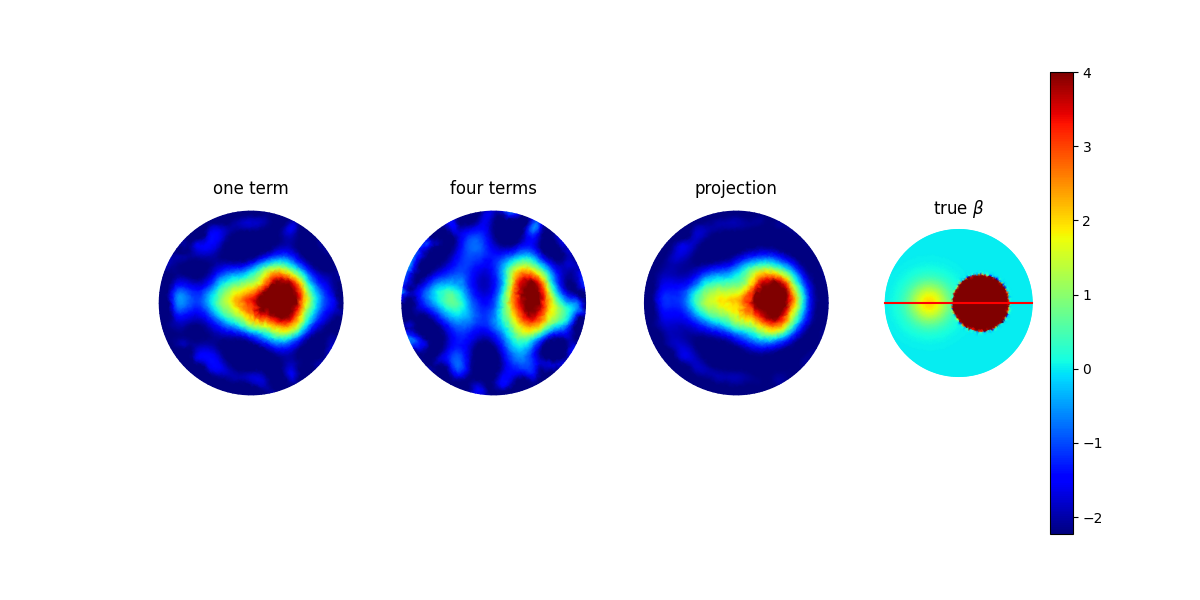}
        \caption{Reconstruction of disk shaped inclusion and a Gaussian}
        \label{fig:sub1}
    \end{subfigure}
    \begin{subfigure}[]{1\textwidth}
        \centering
        \includegraphics[height=2.5in]{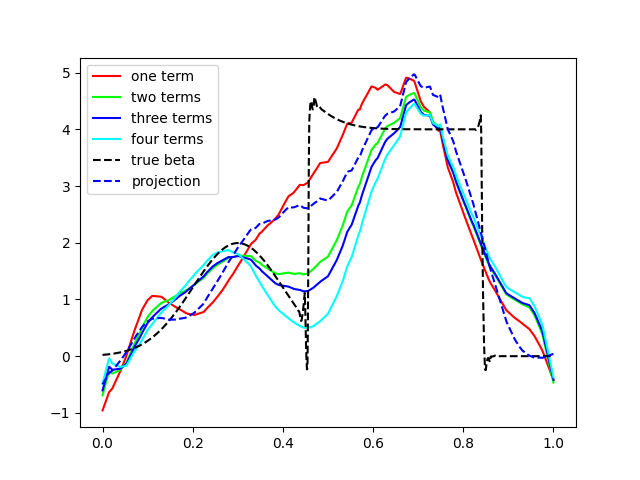}
        \caption{Cross section of reconstruction of disk inclusion and a Gaussian}
        \label{fig:sub2}
\end{subfigure}
    \caption{Reconstruction of $\beta$ with a discontinuous disk shaped inclusion and a Gaussian. Sources were scaled moderately with $g_0=.1$. The series differentiates the two inhomogeneities better than the projection $\mathcal{K}_1 K_1 \beta$.}
    \label{fig:diskandgaussian}
\end{figure}

\section{Discussion}
We have investigated the inverse Born series for scalar waves with polynomial nonlinearities, where the coefficient of the linear term is constant. We have analyzed the convergence of the IBS, and have found that given any contrast and regularization, the IBS will converge if the data is taken to be sufficiently small.  {\color{black}  Numerical simulations demonstrate that even for very high contrast, for sufficiently small scaling, the error in the reconstructions is dominated by the loss of information due to regularization, and the reconstruction is quite close to the projection $\mathcal{K}_1 K_1 \beta$.  However, in some cases, when using a more moderate scaling, the reconstructions appear to be better than the projection. 
%We suspect that due to the nonlinearity, the data $\phi$ may contain higher frequency information than the linear $K_1\beta$, which could lead to the limit of the IBS having higher resolution if the scaling is moderate. 
The explanation of this finding will require further study. Furthermore, the reconstruction results could potentially be improved by using better regularization techniques. }
 
Our results suggest that high contrast nonlinear inhomogeneities {\color{black} of the type (\ref{generalpdeequation}) } will generally be less difficult to reconstruct than linear inhomogeneities. In this light, we suspect that Newton type methods will converge rapidly for small enough data, and for a large class of problems, the inverse Born approximation (the first term in the inverse series) will itself be quite close to $\mathcal{K}_1 K_1 \beta$ for small data.

\section{Acknowledgments}  
 S. Moskow was supported by the NSF grants DMS-2008441 and DMS-2308200. J. Schotland was supported by the NSF grant DMS-1912821 and the AFOSR grant FA9550-19-1-0320.

\end{document}